\documentclass[leqno]{amsart}
\usepackage{times}
\usepackage{amsfonts,amssymb,amsmath,amsgen,amsthm}
\usepackage{hyperref}

\theoremstyle{plain}
\newtheorem{theorem}{Theorem}[section]
\newtheorem{definition}[theorem]{Definition}

\newtheorem{lemma}[theorem]{Lemma}

\newtheorem{proposition}[theorem]{Proposition}

\theoremstyle{remark}
\newtheorem{remark}[theorem]{Remark}

\newtheorem{example}[theorem]{Example}

\def\dis
{\displaystyle}

\def\R{{\mathbf R}}
\def\N{{\mathbf N}}
\def\Sch{{\mathcal S}}
\def\O{\mathcal O}
\def\F{\mathcal F}

\def\({\left(}
\def\){\right)}
\def\<{\left\langle}
\def\>{\right\rangle}
\def\le{\leqslant}
\def\ge{\geqslant}

\def\Tend#1#2{\mathop{\longrightarrow}\limits_{#1\rightarrow#2}}

\def\d{{\partial}}

\def\si{{\sigma}}

\DeclareMathOperator{\RE}{Re}

\numberwithin{equation}{section}

\begin{document}

\title[NLS and frequency saturation]{Nonlinear
  Schr\"odinger equation and frequency saturation}    
\author[R. Carles]{R\'emi Carles}
\address{CNRS \& Univ. Montpellier~2\\ UMR5149\\ Math\'ematiques
\\CC051\\34095 Montpellier\\ France}
\email{Remi.Carles@math.cnrs.fr}

\begin{abstract}
We propose an approach that permits to avoid 
instability phenomena for the nonlinear Schr\"odinger equations. We
show that by approximating the solution in a suitable way, relying on
a frequency cut-off, global
well-posedness is obtained in any Sobolev space with nonnegative
regularity. The error between the exact solution and its approximation
can be measured according to the regularity of the exact solution,
with different accuracy according to the cases considered.
\end{abstract}
\keywords{Nonlinear Schr\"odinger equation; well-posedness;
  approximation}
\thanks{2010 \emph{Mathematics Subject Classification.} {Primary
    35Q55; Secondary 35A01, 35B30, 35B45, 35B65.} }
\thanks{This work was supported by the French ANR project
  R.A.S. (ANR-08-JCJC-0124-01).}
\maketitle

\section{Introduction}
\label{sec:intro}
We consider the nonlinear Schr\"odinger equation 
 \begin{equation}\label{eq:nls}
    i\d_t u+\Delta u= \epsilon|u|^{2\si} u,\quad
    (t,x)\in I\times\R^d  ;\quad u_{\mid t=0}=u_0,
  \end{equation}
for some time interval $I\ni 0$, with $\epsilon =1$ (defocusing case)
or $\epsilon=-1$ (focusing 
case). The aim of this paper is to propose an approach to overcome the
lack of local well-posedness in Sobolev spaces with nonnegative
regularity. 

Recall two important invariances associated to \eqref{eq:nls}:
\begin{itemize}
\item Scaling: if $u$ solves \eqref{eq:nls}, then for $\lambda>0$, so
  does $u_\lambda(t,x):=\lambda ^{1/ \si} u \( \lambda^2 t,
\lambda x \)$. This scaling leaves the $\dot H^{s_c}_x$-norm
invariant, with
$s_c = d/2-1/\si$.
\item Galilean: if $u$ solves \eqref{eq:nls}, then for
  $v\in \R^d$, so does $e^{iv\cdot x -i |v|^2
  t/2}u(t,x-vt)$. This transform leaves the $L^2_x$-norm invariant.
\end{itemize}
These two arguments suggest that the critical Sobolev regularity to
solve \eqref{eq:nls} is $\max(s_c,0)$. 
 Indeed,  if $s_c\ge 0$, local
well-posedness from $H^s(\R^d)$ to $H^s(\R^d)$ for $s\ge s_c$ has been
established in \cite{CW90}, and if $s_c<0$, local
well-posedness from $H^s(\R^d)$ to $H^s(\R^d)$ for $s\ge 0$ has been
established in \cite{TsutsumiL2}. 

If $s_c>0$, pathological phenomena have been exhibited for initial
data in $H^s(\R^d)$ with $0<s<s_c$: Gilles Lebeau has proved a ``norm
inflation'' phenomenon for the wave equation $\d_t^2 u -\Delta
    u +u^p=0$, $x\in 
  \R^3$, $p\in 2\N+1$, $p\ge 7$ (\cite{Lebeau01}; see also
  \cite{MetivierBourbaki}). The analogous result for \eqref{eq:nls}
  has been established in \cite{CCT2} and
  \cite{BGTENS}. 
\begin{theorem}[From \cite{BGTENS,CCT2}]\label{theo:instab}
Let $\si\ge 1$. Assume that $s_c=d/2-1/\si >0$, and let
$0<s<s_c$. 
There exists a family $(u_0^h)_{0<h \le 
  1}$ in ${\mathcal S}({\mathbb R}^d)$ with 
\begin{equation*}
  \|u_0^h\|_{H^{s}({\mathbb R}^d)} \to 0 \text{ as
  }h \to 0, 
\end{equation*}
a solution $u^h$ to \eqref{eq:nls} and $0<t^h \to 0$, such that: 
\begin{equation*}
  \|u^h(t^h)\|_{H^{s}({\mathbb R}^d)} \to +\infty \text{ as }h \to
 0.
\end{equation*}
\end{theorem}
The argument of the proof consists in considering  concentrated
initial data, 
$$u_0(x) = 
h^{s-d/2}(\log 1/h)^{-\alpha} a_0\(\frac{x}{h}\)\quad\text{with }h\to
0,$$
and showing that for very 
short time, the Laplacian can be
neglected in \eqref{eq:nls}. The 
above result then stems from its (easy) counterpart in the ODE case,
by choosing a suitable $\alpha>0$.  
In the spirit of \cite{Lebeau05}, the above result has been  strengthened
to a ``loss of regularity'' in \cite{ACMA,CaARMA,ThomannAnalytic}; the
assumptions and conclusion are similar to that in
Theorem~\ref{theo:instab}, the only difference is that 
$u^h(t^h,\cdot)$ is measured in $H^k(\R^d)$ for any $k>s/(1+\si(s_c-s))$,
thus allowed to be smaller than $s$. In 
all the cases mentioned here, the lack of uniform continuity of the
nonlinear flow map near the origin is due to the appearance of higher
and higher
frequencies on a very short time scale. If $s_c<0$, similar
pathological phenomena have been established in $H^s(\R^d)$ with
$s<0$, where on the contrary, low frequencies are ignited; see
e.g. \cite{BeTa05,CDS-p,CCT2,KPV01}. In the rest of this paper, 
we focus on nonnegative regularity, $s\ge 0$. 
\smallbreak

The goal of this paper is twofold. First, we want to investigate a
method to remove the pathology mentioned above, causing a lack of
well-posedness for \eqref{eq:nls}, in a deterministic way, as opposed
to the probabilistic approach initiated in \cite{BT08a,BT08b} for the
wave equation. The
other motivation is related to numerical simulations for 
\eqref{eq:nls}, where high frequencies may be a source of important
errors; see for instance \cite{IgZu-p}, a reference which will be discussed
further into details in Sections~\ref{sec:smooth} and
\ref{sec:strichartz}. 
\smallbreak

We show that with a
suitable cut-off on the high frequencies of the nonlinearity, the
obstructions to local well-posedness vanish, and the problem becomes
globally well-posed: the nonlinear evolution of any initial datum in
$L^2(\R^d)$ can be
controlled \emph{a priori},  an information which may be useful for
numerics, since we do not have to decide if the initial datum belongs
to a full measure set or not. This strategy is validated inasmuch as this
procedure yields a good approximation of the solution to
\eqref{eq:nls} as the cut-off tends to the identity. Note that this
approach can be viewed as a deterministic counterpart of the one
presented in \cite{BTT-p} (see also
\cite{Bu11}). There, for the 
one-dimensional $L^2$-supercritical defocusing nonlinear Schr\"odinger
equation, the authors construct a Gibbs measure such that, among other
features, the pathological
phenomenon described in Theorem~\ref{theo:instab} occurs for a set of
initial data whose measure is zero: on the support of the Gibbs
measure, the Cauchy problem is globally well-posed, and a scattering
theory is available. Both points of view aim at
showing that norm inflation in the sense of Theorem~\ref{theo:instab}
is a rare phenomenon: in \cite{BTT-p}, the authors give a rigorous
meaning to this statement in an abstract way, while we are rather
interested in a recipe to avoid instabilities for sure, by a suitable
approximation of the equation, which can be used typically for
numerical simulations.
\smallbreak

Our choice of cutting off the high frequencies
  instead of, for instance, the values 
of the function itself is indeed motivated by numerics, where it is
standard to filter out high frequencies (sometimes without even saying
so). In an appendix, we discuss another approach, consisting in
saturating the values of the nonlinearity. One could of course combine both
approaches, frequency and physical saturations. 

\subsection*{Notations}

We define the Fourier transform by the formula
\begin{equation*}
  \widehat f(\xi)=\F(f)(\xi)=\frac{1}{(2\pi)^{d/2}}\int_{\R^d}e^{-ix\cdot
    \xi}f(x)dx,\quad f\in \Sch(\R^d). 
\end{equation*}
We write $a\lesssim b$ if there exists $C$ such that $a\le Cb$. In the
presence of a small parameter $h$, the notation indicates that $C$ is
independent of $h\in (0,1]$.

\section{From instability to global well-posedness}
\label{sec:exam}

Let $\chi:\R^d\to [0,1]$ be a smooth function,
equal to one on the unit ball, and even: $\chi(-x)=\chi(x)$ for all
$x\in \R^d$. It may be compactly supported, in the
Schwartz class $\Sch(\R^d)$, or with a slower decay at
infinity. For simplicity, we will not discuss sharp assumptions on
$\chi$. We define the frequency ``cut-off'' $\Pi$ as the Fourier
multiplier
\begin{equation*}
  \widehat{\Pi(f)}(\xi)=\chi(\xi)\widehat f(\xi).
\end{equation*}
As pointed out in the introduction, in the examples constructed to
prove the lack of local well-posedness, the mechanism of high
frequencies amplification occurs at the level of the ordinary
differential equation. We discuss some strategies to saturate high
frequencies at the ODE level first, with $\epsilon=1$ for simplicity. 

\subsection{Candidates at the ODE level}
\label{sec:ODE}

The first possibility to prevent the appearance of high frequencies by
nonlinear self-interaction consists in saturating the whole nonlinearity:
\begin{equation}\label{eq:I}
  i\d_t v = \Pi\(|v|^{2\si}v\).
\end{equation}
This can be viewed as an extremely simplified version  of the
$I$-method (see e.g. \cite{CKSTT02}). Another choice consists in
saturating the high frequencies in the ``nonlinear multiplicative
potential'' only, that is $|v|^{2\si}$: for $\si\in \N$, we propose
two possibilities,
\begin{align}
  \label{eq:2}
&  i\d_t v = \Pi\(|v|^{2\si}\)v,\\
  \label{eq:3}
&  i\d_t v = \(\Pi\(|v|^{2}\)\)^{\si}v.
\end{align}
In the cubic case $\si=1$, the last two approaches obviously
coincide. These two approaches have two advantages over \eqref{eq:I}:
\begin{itemize}
\item They preserve the gauge invariance. If $v$ solves the equation,
  then so does $v e^{i\theta}$ for any constant $\theta\in \R$. 
\item They preserve the conservation of mass.
\end{itemize}
To see the second point, rewrite $\Pi(f)=K\ast f$, with 
$K(x)=(2\pi)^{-d/2}\widehat \chi(-x)$. Since $\chi$ is even and
real-valued, so is $K$, and therefore $\d_t |v|^2=0$ in \eqref{eq:2}
and \eqref{eq:3}. This identity leads to the conservation of the
$L^2$-norm at the PDE level. 
\smallbreak

Before passing to the PDE case, we conclude this section by showing
that even at the ODE level, cutting off high frequencies in the
initial data does not suffice to prevent the appearance of higher
frequencies in the solution for positive time. For $a\in
\Sch(\R^d)$ and $s>0$, consider $v^h$ the solution to
\begin{equation*}
  i\d_t v^h = |v^h|^{2\si}v^h;\quad v^h(0,x) = h^{s-d/2}a\(\frac{x}{h}\).
\end{equation*}
Then $v^h_{\mid t=0}$ is bounded in $H^s(\R^d)$, uniformly in
$h\in(0,1]$, and if $\widehat a$ is compactly supported (in $B(0,R)$), then
$\widehat{v^h}_{\mid t=0}$  is compactly supported (in
$B(0,R/h)$). Since $\d_t|v^h|^2=0$, we have the explicit formula
\begin{equation*}
  v^h(t,x) = h^{s-d/2}a\(\frac{x}{h}\)\exp\(-it
  h^{2\si(s-d/2)}\left\lvert a\(\frac{x}{h}\)\right\rvert^{2\si}\). 
\end{equation*}
We check that for $t>0$, as $h\to 0$, the homogeneous Sobolev norms behave like
\begin{equation*}
  \|v^h(t)\|_{\dot H^k}\approx h^{s -2k\si (s-d/2) -k}t^k,
\end{equation*}
at least for $k\in \N$. The above quantity is unbounded as $h\to 0$ if
\begin{equation*}
  k>\frac{s}{1+2\si(s-d/2)}. 
\end{equation*}
Therefore, if $s<d/2$, $v^h(t,\cdot)$ is unbounded in $H^s(\R^d)$ for
$t>0$, as $h\to 
0$: cutting off the high frequencies in the initial data does not
suffice to control the frequency support of the solution. On the other hand,
the models \eqref{eq:2} and \eqref{eq:3} prevent the appearance of
high frequencies by nonlinear self-interaction. The above mechanism is
essentially the one that leads to the norm inflation phenomenon in
\cite{BGTENS,CCT2,Lebeau01}, except that in those papers, the
approximation by an ODE is used only on a time interval where the
$H^s$-norm becomes unbounded, but not the $H^k$-norm for any
$k<s$. The above mechanism at the PDE level leads to the loss of
regularity \cite{ACMA,CaARMA,Lebeau05,ThomannAnalytic}, where indeed
$k$ is allowed to be smaller than $s$, as recalled in the
introduction. Roughly speaking, the appearance of oscillations is
quite similar to the above ODE example; in the PDE case, the
numerology is different, and the proof is more intricate.  

\subsection{Choice at the PDE level}
\label{sec:PDE}

We consider now the equations
\begin{equation}\label{eq:edp1}
  i\d_t u + P(D)u = \epsilon \Pi\(|u|^{2\si}\)u,
\end{equation}
and
\begin{equation}
  \label{eq:edp2}
 i\d_t u + P(D)u = \epsilon \(\Pi\(|u|^{2}\)\)^\si u, 
\end{equation}
where $P(D)$ is a Fourier multiplier with a real-valued symbol
$P:\R^d\to \R$,
\begin{equation*}
  \widehat{P(D)f}=P(\xi)\widehat f(\xi). 
\end{equation*}
The $L^2$-norm of $u$ is formally independent of time:
\begin{equation}\label{eq:L2}
  \frac{d}{dt}\int_{\R^d}|u(t,x)|^2dx=0. 
\end{equation}
In view of this conservation and of Young inequality
\begin{equation}\label{eq:young}
\|\Pi(f)\|_{L^\infty}\le \|K\|_{L^\infty}\|f\|_{L^1},
\end{equation}
the option \eqref{eq:edp2} seems more interesting than
\eqref{eq:edp1}, and we have the following result.
\begin{theorem}\label{theo:gwp}
  Let $\si\in \N$, $\epsilon\in \{\pm 1\}$, $P:\R^d\to \R$ and
  $\chi\in \Sch(\R^d)$
  even and real-valued. 
  \begin{itemize}
  \item For any $u_0\in L^2(\R^d)$, \eqref{eq:edp2} has a unique
    solution $u\in C(\R;L^2(\R^d))$ such that $u_{\mid t=0}=u_0$. Its
    $L^2$-norm is independent of time: \eqref{eq:L2} holds. 
\item If in addition $u_0\in H^s(\R^d)$, $s\in \N$, then $u\in
  C(\R;H^s(\R^d))$.
\item The flow map $u_0\mapsto u$ is uniformly continuous from the
  balls in $L^2(\R^d)$ to $C(\R;L^2(\R^d))$. More precisely, for all
  $u_0,v_0\in L^2(\R^d)$, there exists $C$ depending on $\si$,
  $\|K\|_{L^\infty}$, 
  $\|u_0\|_{L^2}$ and $\|v_0\|_{L^2}$ such that for all $T>0$, 
  \begin{equation}\label{eq:continuity}
    \|u-v\|_{L^\infty([-T,T];L^2(\R^d))}\le \|u_0-v_0\|_{L^2(\R^d)}e^{CT},
  \end{equation}
where $u$ and $v$ denote the solutions to \eqref{eq:edp2} with initial
data $u_0$ and $v_0$, respectively. 
\item More generally, let $s\in \N$. For all
  $u_0,v_0\in H^s(\R^d)$, there exists $C$ depending on $\si$,
  $\|K\|_{W^{s,\infty}}$, 
  $\|u_0\|_{H^s}$ and $\|v_0\|_{H^s}$ such that for all $T>0$, 
  \begin{equation}\label{eq:continuity2}
    \|u-v\|_{L^\infty([-T,T];H^s(\R^d))}\le \|u_0-v_0\|_{H^s(\R^d)}e^{CT}.
  \end{equation}
  \end{itemize}
\end{theorem}
\begin{remark}
  As pointed out in \cite{CFH11}, even if the solution is constructed
  by a fixed point argument, the continuity of the flow map is not
  trivial in general. In the case of Schr\"odinger
equations \eqref{eq:nls}, continuity of the flow map in
$H^s(\R^d)$  is known only in a limited
number of cases:
see \cite{TsutsumiL2} for $s=0$, \cite{Kato87} for $s=1$ and $s=2$, and
\cite{CFH11} for $0<s<1$.
 \end{remark}
\begin{proof}
First, recall that $S(t)= e^{-itP(D)}$ is a unitary group on $\dot
H^s(\R^d)$, $s\in \R$. Duhamel's formula associated to \eqref{eq:edp2}
reads
\begin{equation}\label{eq:duhamel}
  u(t)=S(t)u_0-i\epsilon \int_0^t S(t-\tau)\(\(K\ast |u|^2\)^\si
  u\)(\tau)d\tau.
\end{equation}
  The local existence in $L^2$ stems from a standard fixed point
  argument in
  \begin{equation*}
    X(T) = \{u \in C([-T,T];L^2(\R^d));\quad
    \|u\|_{L^\infty([-T,T];L^2)}\le 2\|u_0\|_{L^2}\}. 
  \end{equation*}
Denote by $\Phi(u)(t)$ the right hand side of \eqref{eq:duhamel}. In
view of \eqref{eq:young}, for $t\in [-T,T]$,
\begin{align*}
  \|\Phi(u)(t)\|_{L^2}&\le \|u_0\|_{L^2} +\int_{-T}^T \left\|\(\(K\ast |u|^2\)^\si
  u\)(\tau)\right\|_{L^2}d\tau \\
&\le \|u_0\|_{L^2}+\int_{-T}^T  \left\|K\ast |u(\tau)|^2\right\|_{L^\infty}^\si
  \|u(\tau)\|_{L^2}d\tau \\
&\le \|u_0\|_{L^2}+\|K\|_{L^\infty}^\si \int_{-T}^T  \|u(\tau)\|_{L^2}^{2\si+1}
  d\tau .
\end{align*}
By choosing $T>0$ sufficiently small, we see that $X(T)$ is stable
under the action of $\Phi$. Note that in the case of the model
\eqref{eq:edp1}, the above estimate would have to be adapted, forcing
us to work in a space smaller than $X(T)$ ($L^2$ regularity in space
would no longer be sufficient in general). Contraction is established
in the same way:
\begin{align*}
  \|\Phi(u)(t)-\Phi(v)(t)\|_{L^2} & \le \int_{-T}^T \left\|\(\(K\ast |u|^2\)^\si
  u\)(\tau)- \(\(K\ast |v|^2\)^\si
  v\)(\tau)\right\|_{L^2}d\tau \\
& \le \int_{-T}^T \left\|\(\(K\ast |u|^2\)^\si- \(K\ast
  |v|^2\)^\si\)u\right\|_{L^2}d\tau\\
&\quad +\int_{-T}^T \left\|\(\(K\ast |v|^2\)^\si\)(u-v)\right\|_{L^2}d\tau .
\end{align*}
Using the estimate $\lvert  a^\si-
b^\si\rvert\lesssim (|a|^{\si-1}+|b|^{\si-1})|a-b|$, and
\eqref{eq:young} again, we infer
\begin{align*}
  \|\Phi(u)(t)-\Phi(v)(t)\|_{L^2} & \lesssim \|K\|_{L^\infty}^\si
\int_{-T}^T\( \|u\|_{L^2}^{2\si-1}
+\|v\|_{L^2}^{2\si-1}\)\|u-v\|_{L^2}\|u\|_{L^2}d\tau \\ 
&\quad +\|K\|_{L^\infty}^\si
\int_{-T}^T \|v\|_{L^2}^{2\si}\|u-v\|_{L^2}d\tau ,
\end{align*}
where all the functions inside the integrals are implicitly evaluated
at time $\tau$. Choosing $T>0$ possibly smaller, $\Phi$ is a
contraction on $X(T)$. Note that this small time $T$ depends only on
$\si$, $\|K\|_{L^\infty}$ and $\|u_0\|_{L^2}$. Since the $L^2$-norm of
$u$ is preserved (see e.g. \cite{CazCourant} for the rigorous
justification), the construction of a local solution can be repeated
indefinitely, hence global existence and uniqueness at the $L^2$
level.
\smallbreak

Global existence in $H^s(\R^d)$ for $s\in \N$ then follows easily,
thanks to the estimate
\begin{equation*}
 \left\|\(K\ast |u|^2\)^\si
  u\right\|_{H^s}\lesssim
\sum_{|\alpha|+|\beta|=s}\left\|\d^\alpha\(K\ast |u|^2\)^\si 
  \d^\beta u\right\|_{L^2}\lesssim \|K\|_{W^{s,\infty}}^\si
\|u\|_{L^2}^\si \|u\|_{H^s}. 
\end{equation*}
The continuity of the flow map in $L^2$ is obtained by resuming the
estimate written to establish the contraction of $\Phi$: for $t>0$,
\begin{align*}
  \|u(t)-v(t)\|_{L^2}&\le \|u_0-v_0\|_{L^2} + \|K\|_{L^\infty}^\si
\int_0^t
\(\|u\|_{L^2}^{2\si}+\|v\|_{L^2}^{2\si}\)\|u-v\|_{L^2}d\tau\\
&\le \|u_0-v_0\|_{L^2} +
\|K\|_{L^\infty}^\si\(\|u_0\|_{L^2}^{2\si}+\|v_0\|_{L^2}^{2\si}\) 
\int_0^t\|u-v\|_{L^2}d\tau,
\end{align*}
where we have used the conservation of the $L^2$-norm. Proceeding
similarly for $t<0$, Gronwall lemma
then yields \eqref{eq:continuity} for $C$ depending only of $\si$,
$\|K\|_{L^\infty}$, $\|u_0\|_{L^2}$ and $\|v_0\|_{L^2}$. Finally,
\eqref{eq:continuity2} is obtained in a similar fashion.
\end{proof}
\begin{remark}
  The proof of continuity of the flow map is extremely easy.
  This is
  in sharp constrast with the case of the equation without frequency
  cut-off. In the case of Schr\"odinger
equations ($P(\xi)=-|\xi|^2$), continuity is more intricate to
establish (see \cite{TsutsumiL2}), and is true only for
$L^2$-subcritical nonlinearities, $\si\le 2/d$, from \cite{CCT2}. 
\end{remark}
We note that even for large $\si$, global well-posedness in $L^2$ is
available, in sharp contrast with the nonlinear Schr\"odinger equation
\eqref{eq:nls}. Even in the focusing case $\epsilon=-1$, the high
frequency cut-off prevents finite time blow-up. In
\eqref{eq:continuity2}, consider $v_0=v=0$ and $s=1$ for instance: by
comparison with the case of \eqref{eq:nls}, we see that the constant
$C$ necessarily depends on $K$ (or equivalently on $\chi$), and is
unbounded as $\chi$ converges to the Dirac mass. The frequency
cut-off $\Pi$ removes the instabilities, and prevents finite time
blow-up.

\begin{remark}[Hamiltonian structure in the cubic case]
 If $\si=1$, \eqref{eq:edp1} and \eqref{eq:edp2} coincide. 
We have the equivalence
\begin{equation*}
  \chi \text{ 
even and real-valued }\Longleftrightarrow K \text{ 
even and real-valued.}
\end{equation*}
This implies that under the assumption of Theorem~\ref{theo:gwp},
\eqref{eq:edp2} has 
an Hamiltonian structure, and the conserved energy is
\begin{equation*}
  H(u) = \int_{\R^d}\overline u(x) P(D)u(x)dx +
    \frac{\epsilon}{2} 
    \iint K(x-y)|u(y)|^2|u(x)|^2 dxdy.
\end{equation*} 
\end{remark}

\section{Convergence in the smooth case}
\label{sec:smooth}

Suppose that $P(D)$ converges to $\Delta$ and that $\Pi$  converges to
${\rm Id}$: does the
 solution to \eqref{eq:edp2} converge  to the solution of NLS? We show that
this is the case under suitable assumptions on these convergences, at
least in the case where the solution to the limiting equation
\eqref{eq:nls} is very smooth. In the sequel, the convergence is
indexed by $h\in (0,1]$. 
\begin{proposition}\label{prop:CVsmooth}
 Let $\si\in \N$. We assume that $P$ and $\Pi$ verify the following
 properties: 
\begin{itemize}
\item There exist $\alpha,\beta\ge 0$ such that $P_h(\xi) = -|\xi|^2 +
    \O\(h^\alpha\<\xi\>^\beta\)$.
\item $\chi_h(\xi)=\chi\(h\xi\)$, with $\chi\in
    \Sch(\R^d;[0,1])$ even, 
  real-valued, $\chi=1$  on the unit ball. 
\end{itemize}
Denote by $u^h$ the solution to \eqref{eq:edp2} with
 $P_h$ and $\chi_h$, such that $u^h_{\mid t=0}=u_{\mid t=0}$. Suppose
 that the solution to \eqref{eq:nls} satisfies $u \in
 L^\infty([0,T];H^{s+\beta})$, for some  
 $s>d/2$. Then
  \begin{equation*}
     \|u-u^h\|_{L^\infty([0,T];H^{s})}\lesssim h^{\min(\alpha,\beta)}.
  \end{equation*}
\end{proposition}
\begin{example} The above assumption on $P_h$ is
  satisfied with $\alpha=1$ and $\beta=2$ in the following cases:
  \begin{itemize}
  \item $P_h(\xi) = \dis \frac{-|\xi|^2}{1+h|\xi|^2}$.
\item $P_h(\xi) =\dis-\frac{1}{h}\arctan\(h|\xi|^2\)$.
  \end{itemize}
The second example is borrowed from \cite{DeFa09}, where this
truncated operator appears naturally when discretizing the Laplacian
for numerical schemes. 
\end{example}
\begin{remark}
  In this result, no assumption is needed on the possible decay
  of $\chi$ at infinity. 
\end{remark}
\begin{proof}
   Let $w^h = u-u^h$:  it satisfies $w^h_{\mid t=0}=0$ and 
\begin{align*}
  i\d_t w^h + P_h(D)w^h & =
\epsilon \(\Pi_h\(|u|^2\)\)^\si u - \epsilon\(\Pi_h\(|u^h|^2\)\)^\si u^h \\
&\quad+
\(P_h(D)-\Delta\)u + \epsilon \(|u|^{2\si}-\(\Pi_h\(|u|^2\)\)^\si \)u,
\end{align*}
where we have denoted by $\Pi_h$ the Fourier
  multiplier of symbol $\chi_h$. 
Denote by $R^h(u)$ the second line, which corresponds to a source
term. In view of the assumption on $P_h$, there exists $C$ independent
of $h\in (0,1]$ such that
\begin{equation*}
  \| P_h(D)f-\Delta f \|_{H^s}\le C h^\alpha
      \|f\|_{H^{s+\beta}}\quad \forall f\in H^{s+\beta}(\R^d). 
\end{equation*}
We also have, by Plancherel formula,
\begin{align*}
  \|\(1-\Pi_h\)f\|_{H^s}^2 & = \int_{\R^d}
  \(1-\chi(h\xi)\)^2\<\xi\>^{2s}|\widehat f(\xi)|^2d\xi\\
&\le \int_{|\xi|>1/h}\<\xi\>^{2s}|\widehat f(\xi)|^2d\xi\\
&\le h^{2\beta}\int_{|\xi|>1/h}\<\xi\>^{2s+2\beta}|\widehat
f(\xi)|^2d\xi\le h^{2\beta}\|f\|_{H^{s+\beta}}^2.
\end{align*}
Therefore,
\begin{equation*}
  \|R^h(u)\|_{L^\infty([0,T];H^s)}\lesssim h^{\min
    (\alpha,\beta)}\|u\|_{L^\infty([0,T];H^{s+\beta})} .
\end{equation*}
Now since $s>d/2$, $H^{s}(\R^d)$ is an algebra, and there exists $C$
independent of $h$ such that
\begin{equation*}
  \left\| \(\Pi_h\(|u|^2\)\)^\si u - \(\Pi_h\(|u^h|^2\)\)^\si
    u^h \right\|_{H^s}\le C \|\widehat \chi\|_{L^1}^\si \(\|u\|_{H^s}^{2\si}+
  \|u^h\|_{H^s}^{2\si}\)\|u-u^h\|_{H^s},
\end{equation*}
where the Young inequality that we have used is not the same as in
Section~\ref{sec:exam}:
\begin{equation*}
  \|K\ast f\|_{L^2}\le \|K\|_{L^1}\|f\|_{L^2}. 
\end{equation*}
This is essentially the only way to obtain an estimate independent of
$h\in (0,1]$. Indeed, $\Pi_h(f)=K_h\ast f$, with 
\begin{equation*}
  K_h(x)=\frac{1}{(2\pi)^{d/2} h^d}\widehat \chi\(\frac{-x}{h}\). 
\end{equation*}
The result then stems from a bootstrap argument: so long as
$$\|u^h\|_{L^\infty([0,t]; H^s)}\le 1+ \|u\|_{L^\infty([0,T]; H^s)},$$
Gronwall lemma yields
\begin{equation*}
\|u-u^h\|_{L^\infty([0,t];H^s)}  \lesssim h^{\min
    (\alpha,\beta)}\|u\|_{L^\infty([0,T];H^{s+\beta})}.
\end{equation*}
Therefore, up to choosing $h$ sufficiently small, this estimate is
valid up to $t=T$. 
\end{proof}
Such a convergence result can be compared to the one proved in
\cite{IgZu-p} to prove the convergence of numerical
approximations. The approach there is a bit different though, inasmuch
as the frequency cut-off does not affect the nonlinearity (as in
\eqref{eq:edp2}), but the initial data: consider $v^h$ solution to 
\begin{equation*}
  i\d_t v^h +P_h(D)v^h = \epsilon |v^h|^{2\si}v^h;\quad v^h_{\mid
    t=0}=\Pi_h u_0. 
\end{equation*}
Then in \cite{IgZu-p}, the discrete analogue to $\Pi_h u-v_h$ is
proven to be small. Proposition~\ref{prop:CVsmooth} differs from the
results in \cite{IgZu-p} on several aspects:
\begin{itemize}
\item The context in \cite{IgZu-p} is discrete.
\item Only the low frequency part of $u$, $\Pi_h u$, is shown to be
  well approximated. 
\item The regularity assumption on $u$ may be much weaker. 
\end{itemize}
As mentioned above, the second point is due to the choice of the
model. However, as discussed in Section~\ref{sec:ODE}, controlling the
high frequencies of the initial data must not be expected to ensure a
control of high frequencies of the solution $v^h$ for positive time. 
\smallbreak

The third point is due to the use of Strichartz estimates in
\cite{IgZu-p}. In the next section, we show that in the presence of
dispersion (with $P_h(\xi)=-|\xi|^2$), Proposition~\ref{prop:CVsmooth}
can be adapted to rougher data.

\section{Convergence using dispersive estimates}
\label{sec:strichartz}

We first recall the standard definition.
\begin{definition}\label{def:adm}
 A pair $(p,q)\not =(2,\infty)$ is admissible if $p\ge 2$, $q\ge  2$,
 and 
$$\frac{2}{p}= d\left( \frac{1}{2}-\frac{1}{q}\right).$$
\end{definition}
We shall consider \eqref{eq:edp2} when $P(D)$ is exactly the
Laplacian, and not an approximation as in
Proposition~\ref{prop:CVsmooth}. The reason is that when $P$ is
bounded, then no Strichartz estimate is available, as we now recall. 
Let $S(\cdot)$ be bounded on $H^s$ for all $s\ge 0$. By Sobolev
embedding, for all $(p,q)$ (not necessarily admissible) 
with $2\le q<\infty$, there exists $C>0$ such that for all $u_0\in
H^{d/2-d/q}(\R^d)$, and all finite time interval $I$,
\begin{align*}
    \|S(\cdot)u_0\|_{L^p(I;L^q(\R^d))}&\le
    C\|S(\cdot)u_0\|_{L^p(I;H^{d/2-d/q}(\R^d))}\\
    &\le C\|u_0\|_{L^p(I;H^{d/2-d/q}(\R^d))} =
    C|I|^{1/p}\|u_0\|_{H^{d/2-d/q}(\R^d)}. 
  \end{align*}
If the Fourier multiplier $P$
is bounded, the above estimate cannot be improved, in sharp contrast
with the result provided by Strichartz estimates.
\begin{proposition}[From \cite{CaDisp}]
  Let $d\ge 1$, and $P\in L^\infty(\R^d;\R)$. Denote 
  $S(t)=e^{-itP(D)}$. Suppose that there
  exist an admissible pair $(p,q)$, an index $k\in \R$, a time
  interval $I\ni 0$, $|I|>0$, and  a constant $C>0$ such that
  \begin{equation*}
    \|S(\cdot)u_0\|_{L^p(I;L^q(\R^d))}\le C\|u_0\|_{H^k(\R^d)},\quad
    \forall u_0\in H^k(\R^d). 
  \end{equation*}
Then necessarily, $k\ge 2/p= d/2-d/q$.  
\end{proposition}
We now state the main result of this section.
\begin{theorem}\label{theo:CVstrichartz}
Let $\si\in \N$ and $T>0$. We assume that
$\chi_h(\xi)=\chi\(h\xi\)$, with $\chi\in \Sch(\R^d)$ even,
  real-valued, $\chi=1$  on $B(0,1)$. Let $u$ solve \eqref{eq:nls}, and
  consider the solution $u^h$ to
  \begin{equation*}
   i\d_t u^h +\Delta u^h = \epsilon \(\Pi_h \(|u^h|^2\)\)^\si
   u^h;\quad u^h_{\mid t=0}=u_0.
  \end{equation*}
$1.$ Suppose that $\si=1$ and $d\le 2$.
If $u \in L^\infty([0,T];L^2)\cap
L^{8/d}([0,T];L^4)$, then   
  \begin{equation*}
     \|u-u^h\|_{L^\infty([0,T];L^2)}\Tend h 0 0 . 
  \end{equation*}
$2.$ Suppose that $\si=1$ and $d=3$.
\begin{itemize}
\item If $u,\nabla u \in
  L^\infty([0,T];L^2)\cap 
L^{8/d}([0,T];L^4)$, then
  \begin{equation*}
     \|u-u^h\|_{L^\infty([0,T];H^1)} \Tend h 0 0 . 
  \end{equation*}
\item If $u \in L^\infty([0,T];H^{s})$, with $s>3/2$,
then
 \begin{equation*}
    \|u-u^h\|_{L^\infty([0,T];L^2)}\lesssim h^{s}\quad \text{and}\quad
   \|u-u^h\|_{L^\infty([0,T];H^{1})}\lesssim h^{s-1}.
  \end{equation*}
\end{itemize}
$3.$ Suppose that $\si\ge 1$ and $d\le 2$. 
 If $u \in L^\infty([0,T];H^{s})$, with $s\ge 1$ and $s>d/2$,
  then 
  \begin{equation*}
    \|u-u^h\|_{L^\infty([0,T];L^2)}\lesssim h^{s}\quad \text{and}\quad
   \|u-u^h\|_{L^\infty([0,T];H^{1})}\Tend h 0 0.
  \end{equation*}
If in addition $s>1$, then 
\begin{equation*}
  \|u-u^h\|_{L^\infty([0,T];H^{1})}\lesssim h^{s-1}. 
\end{equation*}
\end{theorem}
\begin{remark}
  Suppose $u_0$ sufficiently smooth. If $\epsilon =+1$ (defocusing
  case), the  bounds for $u$ are 
  known in several cases, with $T>0$ arbitrarily large. On the
  contrary, if $\epsilon=-1$ (focusing case), $T$ may have to be
  finite, bounded by a blow-up time. See
  e.g. \cite{CazCourant,GV84}. Typically, if $\si=d=1$, then the
  assumption of the first point is fulfilled for all $T>0$ as soon as $u_0\in
  L^2(\R)$, for $\epsilon\in \{\pm 1\}$, from \cite{TsutsumiL2}, and
  if $\si \ge 1$, $d\le 2$, the assumption of the third point is
  fulfilled for all $T>0$ as soon as $u_0\in 
  H^s(\R^d)$, for $\epsilon=+1$, from \cite{GV84}.
\end{remark}
\begin{proof}
For fixed $h>0$, Theorem~\ref{theo:gwp} shows that $u^h\in
C(\R;H^{k})$, with $k=0,1$ or $s$ according to the cases considered in
the assumptions of the theorem. Of course, the bounds provided by
Theorem~\ref{theo:gwp} blow up as $h\to 0$ if $k>0$. 
 
As in the proof of Proposition~\ref{prop:CVsmooth}, let $w^h=u-u^h$.
The equation satisfied by $w^h$ is simpler than in the proof of
Proposition~\ref{prop:CVsmooth}, since $P_h(D)=\Delta$:
\begin{equation*}
  i\d_t w^h + \Delta w^h  =
\epsilon \(\Pi_h\(|u|^2\)\)^\si u - \epsilon\(\Pi_h\(|u^h|^2\)\)^\si u^h 
+ \epsilon \(|u|^{2\si}-\(\Pi_h\(|u|^2\)\)^\si \)u.
\end{equation*}
Resume the notations $R^h(u)=\epsilon
\(|u|^{2\si}-\(\Pi_h\(|u|^2\)\)^\si \)u$ and  $\Pi_h(f)=K_h\ast f$,
with  
$  K_h(x)=(2\pi)^{-d/2} h^{-d}\widehat \chi(-x/h)$. From young
inequality, we have, for all $q\in [1,\infty]$,
\begin{equation}\label{eq:Youngh}
  \|\Pi_h(f)\|_{L^q}\le \|K_h\|_{L^1}\|f\|_{L^q} \le
  \|\widehat\chi\|_{L^1} \|f\|_{L^q}, 
\end{equation}
an estimate which is uniform in $h>0$. 
Introduce the Lebesgue exponents 
\begin{equation*}
  q=2\si+2\quad ;\quad p=\frac{4\si+4}{d\si}\quad ;\quad
  \theta=\frac{2\si(2\si+2)}{2-(d-2)\si}. 
\end{equation*}
The pair $(p,q)$ is admissible, and
\begin{equation}\label{eq:indices}
  \frac{1}{q'}=\frac{2\si}{q}+\frac{1}{q}\quad ;\quad
  \frac{1}{p'}=\frac{2\si}{\theta} +\frac{1}{p}. 
\end{equation}
For $t>0$, denote $L^j_tL^k=L^j([0,t];L^k(\R^d))$. From Strichartz
estimates (see e.g. \cite{CazCourant}),  
\begin{align*}
  \|w^h\|_{L^p_tL^q\cap L^\infty_tL^2}&\lesssim
  \left\|\(\Pi_h\(|u|^2\)\)^\si u - \(\Pi_h\(|u^h|^2\)\)^\si u^h
  \right\|_{L^{p'}_tL^{q'}} + \left\|R^h(u)\right\|_{L^{p_1'}_tL^{q_1'}}\\
&\lesssim \(\|u\|_{L^\theta_t L^q}^{2\si} +\|u^h\|_{L^\theta_t
  L^q}^{2\si}  \)\|w^h\|_{L^p_tL^q} + \left\|R^h(u)\right\|_{L^{p_1'}_tL^{q_1'}},
\end{align*}
where we have used H\"older inequality and \eqref{eq:Youngh}, and
where $(p_1,q_1)$ is an admissible pair whose value will be
given later. 

If $\si=1$ and $d\le 2$, then $\theta\le p$, and we infer 
\begin{equation*}
 \|w^h\|_{L^p_tL^q\cap L^\infty_tL^2} \lesssim t^{1/\theta-1/p}\(\|u\|_{L^p_t
   L^q}^{2\si} +\|u^h\|_{L^p_t 
  L^q}^{2\si}  \)\|w^h\|_{L^p_tL^q} +
\left\|R^h(u)\right\|_{L^{p_1'}_tL^{q_1'}}. 
\end{equation*}
In the first case of the theorem, we assume $u\in L^p([0,T];L^q)$,
since $p=8/d$ and $q=4$ for $\si=1$. We
use again a 
bootstrap argument: so long as $\|u^h\|_{L^p_t 
  L^q}\le 2 \|u\|_{L^p_t   L^q}$, we divide the interval $[0,T]$ into
finitely many small intervals so the first term of the right hand side
is absorbed by the left hand side (recall that $p$ is finite), and we have
\begin{equation*}
  \|w^h\|_{L^p_tL^q\cap L^\infty_tL^2} \lesssim
  \left\|R^h(u)\right\|_{L^{p_1'}_tL^{q_1'}}. 
\end{equation*}
The bootstrap argument is validated provided that
$\|R^h(u)\|_{L^{p_1'}_TL^{q_1'}} \to 0 $ as $h\to 0$. 

If we have only $\si<\frac{2}{d-2}$, then by Sobolev embedding,
\begin{equation*}
  \|u\|_{L^\theta_t L^q}\le t^{1/\theta}\|u\|_{L^\infty_t H^1}. 
\end{equation*}
In the same way as above,
\begin{align*}
  \|\nabla w^h\|_{L^p_tL^q\cap L^\infty_tL^2}& \lesssim
  \left\|\nabla\(\(\Pi_h\(|u|^2\)\)^\si u - \(\Pi_h\(|u^h|^2\)\)^\si u^h \)
  \right\|_{L^{p'}_tL^{q'}} \\
&\quad + \left\|\nabla R^h(u)\right\|_{L^{p_1'}_tL^{q_1'}}
\end{align*}
The first term of the right hand side is controlled by
\begin{equation}\label{eq:nablaw}
\begin{aligned}
  &\left\|\(\Pi_h\(|u|^2\)\)^\si \nabla u - \(\Pi_h\(|u^h|^2\)\)^\si
    \nabla u^h 
  \right\|_{L^{p'}_tL^{q'}}\\
 +&
\left\| u \nabla\(\Pi_h\( |u|^2\)\)^\si  - u^h\nabla \(\Pi_h\(|u^h|^2\)\)^\si
    \right\|_{L^{p'}_tL^{q'}}
\end{aligned}
\end{equation}
Introducing the factor $\(\Pi_h\(|u|^2\)\)^\si \nabla u^h$, the first
term is estimated by
\begin{align*}
&  \left\|\(\Pi_h\(|u|^2\)\)^\si \nabla w^h 
  \right\|_{L^{p'}_tL^{q'}} + \left\|\(\(\Pi_h\(|u|^2\)\)^\si -
    \(\Pi_h\(|u^h|^2\)\)^\si\) 
    \nabla u^h  \right\|_{L^{p'}_tL^{q'}}\\
&\lesssim \left\|\Pi_h\(|u|^2\)\right\|_{L^{\theta/2}_t L^{q/2}}^\si \|\nabla w^h 
  \|_{L^{p}_tL^{q}}\\
&\quad
  +\(\|u\|_{L^\theta_tL^q}^{2\si-2}+\|u^h\|_{L^\theta_tL^q}^{2\si-2}\)
  \left\||u|^2-|u^h|^2\right\|_{L^{\theta/2}_t L^{q/2}}\|\nabla u^h 
  \|_{L^{p}_tL^{q}}\\
&\lesssim \left\|u\right\|_{L^{\theta}_t L^{q}}^{2\si} \|\nabla w^h 
  \|_{L^{p}_tL^{q}}
  +\(\|u\|_{L^\theta_tL^q}^{2\si-1}+\|u^h\|_{L^\theta_tL^q}^{2\si-1}\)
  \left\|w^h\right\|_{L^{\theta}_t L^{q}}\|\nabla u^h 
  \|_{L^{p}_tL^{q}}\\
&\lesssim t^{2\si/\theta}\left\|u\right\|_{L^{\infty}_t H^1}^{2\si} \|\nabla w^h 
 \|_{L^{p}_tL^{q}}\\
&\quad 
  +t^{2\si/\theta}\(\|u\|_{L^\infty_tH^1}^{2\si-1}+\|u^h\|_{L^\infty_tH^1}^{2\si-1}\)
  \left\|w^h\right\|_{L^\infty_t H^1}\|\nabla u^h 
  \|_{L^{p}_tL^{q}}
\end{align*}
Proceeding similarly for the other term in \eqref{eq:nablaw},
splitting $[0,T]$ into finitely many time intervals where the terms
containing $w^h$ on the right hand side can be absorbed by the left
hand side, and using a bootstrap argument,  we end up with 
\begin{equation*}
  \|w^h\|_{L^p_tW^{1,q}\cap L^\infty_tH^1}\lesssim\left\|
  R^h(u)\right\|_{L^{p_1'}_tW^{1,q_1'}}.
\end{equation*}
Therefore, it suffices to show that for some admissible pair
$(p_1,q_1)$, the source term converges to $0$
in $L^{p_1'}([0,T];L^{q_1'})$ (if $\si=1$ and $d\le 2$) or in
$L^{p_1'}([0,T];W^{1,q_1'})$ (in the other cases), so the bootstrap
argument is completed. In addition, the rate of converge of the source
term, if any, yields a rate of convergence for $w^h$.
The theorem then stems from the following lemma, in which $(p,q)$ is
given by \eqref{eq:indices}. 
\begin{lemma}\label{lem:final}
  Let $T>0$. The source term $R^h(u)$ can be controlled as follows.\\
$1.$ Suppose that $\si=1$ and $d\le 2$.
If $u \in L^\infty([0,T];L^2)\cap
L^{8/d}([0,T];L^4)$, then 
  \begin{equation*}
     \|R^h(u)\|_{L^{p'}([0,T];L^{q'})}\Tend h 0 0 . 
  \end{equation*}
$2.$ Suppose that $\si=1$ and $d=3$.
\begin{itemize}
\item If $u,\nabla u \in
  L^\infty([0,T];L^2)\cap 
L^{8/d}([0,T];L^4)$, then
  \begin{equation*}
     \|R^h(u)\|_{L^{p'}([0,T];W^{1,q'})} \Tend h 0 0 . 
  \end{equation*}
\item If $u \in L^\infty([0,T];H^{s})$, with $s>3/2$,
then
 \begin{equation*}
    \|R^h(u)\|_{L^1([0,T];L^2)}\lesssim h^{s}\quad \text{and}\quad
   \|R^h(u)\|_{L^1([0,T];H^{1})}\lesssim h^{s-1}.
  \end{equation*}
\end{itemize}
$3.$ Suppose that $\si\ge 1$ and $d\le 2$. 
 If $u \in L^\infty([0,T];H^{s})$, with $s\ge 1$ and $s>d/2$,
  then 
  \begin{equation*}
    \|R^h(u)\|_{L^1([0,T];L^2)}\lesssim h^{s}\quad \text{and}\quad
   \|R^h(u)\|_{L^1([0,T];H^{1})}\Tend h 0 0.
  \end{equation*}
If in addition $s>1$, then 
\begin{equation*}
  \|R^h(u)\|_{L^1([0,T];H^{1})}\lesssim h^{s-1}. 
\end{equation*}
\end{lemma}
\begin{proof}[Proof of Lemma~\ref{lem:final}]
For the first case, we use H\"older inequality, in view of
\eqref{eq:indices}:
\begin{align*}
  \|R^h(u)\|_{L^{p'}_TL^{q'}} & = \left\|
    \(1-\Pi_h\)\(|u|^2\)u\right\|_{L^{p'}_TL^{q'}} 
\le \left\|
    \(1-\Pi_h\)\(|u|^2\)\right\|_{L^{\theta/2}_TL^{q/2}}\left\|
    u\right\|_{L^{p}_TL^{q}}.
\end{align*}
We note that for $\si=1$, $q=4$, so by Plancherel Theorem,
\begin{equation*}
  \|\(1-\Pi_h\)(|u|^2)\|_{L^2}^2  = \int_{\R^d}
  \(1-\chi(h\xi)\)^2|\F(|u|^2)(\xi)|^2d\xi\le
  \int_{|\xi|>1/h}|\F(|u|^2) (\xi)|^2d\xi.
\end{equation*}
By assumption, $u \in L^p([0,T];L^4)\subset L^\theta([0,T];L^4)$,
thus $|u|^2 \in L^{\theta/2}([0,T];L^2)$, and by Plancherel
Theorem, $\F(|u|^2)\in L^{\theta/2}([0,T];L^2)$. The first point
of the lemma then stems from the Dominated Convergence Theorem.
\smallbreak

For the first case of the second point, we note that now $\theta>p$,
so the above argument must be adapted,
and we have to estimate the gradient of $R^h(u)$ in the same space as
above. Since
 $ L^\infty([0,T];H^1(\R^3))\subset
L^\theta([0,T];L^4(\R^3))$, the Dominated Convergence Theorem yields 
\begin{equation*}
  \|R^h(u)\|_{L^{p'}_TL^{q'}}\Tend h 0 0. 
\end{equation*}
We now estimate $\nabla R^h(u)$. Write
\begin{align*}
  \|\nabla R^h(u)\|_{L^{p'}_TL^{q'}}&\le 
\left\|(1-\Pi_h)\(|u|^2\)
\right\|_{L^{\theta/2}_T 
  L^2}\|\nabla u\|_{L^p_T L^2} \\
&\quad + \left\| (1-\Pi_h)\nabla\(|u|^2\) \right\|_{L^{(1/\theta+1/p)^{-1}}_T
  L^2}\|u\|_{L^\theta_T L^2}\\
&\lesssim
\left\|(1-\Pi_h)\(|u|^2\)
\right\|_{L^{\infty}_T 
  L^2}\|\nabla u\|_{L^p_T L^2} \\
&\quad + \left\| (1-\Pi_h)\nabla\(|u|^2\) \right\|_{L^{(1/\theta+1/p)^{-1}}_T
  L^2}\|u\|_{L^\infty_T L^2}
\end{align*}
By the same argument as above, 
\begin{equation*}
  \left\|(1-\Pi_h)\(|u|^2\)
\right\|_{L^{\infty}_T 
  L^2}\|\nabla u\|_{L^p_T L^2}\Tend h 0 0 . 
\end{equation*}
We note that $u$ bounded in
$L^\infty([0,T];H^1(\R^3))\subset L^\theta([0,T];L^4(\R^3))$, and
$\nabla u$ bounded in $L^p_TL^4$, so $\nabla |u|^2$ is bounded in
$L^{(1/\theta+1/p)^{-1}}_T  L^2$. Invoking Plancherel Theorem and the
Dominated Convergence Theorem like above, we infer
\begin{equation*}
  \left\| (1-\Pi_h)\nabla\(|u|^2\) \right\|_{L^{(1/\theta+1/p)^{-1}}_T
  L^2}\|u\|_{L^\infty_T L^2}\Tend h 0 0 . 
\end{equation*}
This completes the proof for the first case of the second point. 
\smallbreak

For the remaining cases, 
we use that $H^s(\R^d)$ is 
embedded into $L^\infty(\R^d)$: for fixed $t$,
\begin{align*}
  \|R^h(u)(t)\|_{L^2}&\lesssim
  \(\|u(t)\|_{L^\infty}^{2\si-2}+\|\Pi_h(|u(t)|^2)\|_{L^\infty}^{\si-1}\)
\|(1-\Pi_h)(|u(t)|^2 )\|_{L^2}\|u(t)\|_{L^\infty} \\
&\lesssim \|u(t)\|_{L^\infty}^{2\si-1}\|(1-\Pi_h)(|u(t)|^2 )\|_{L^2}
\lesssim \|u(t)\|_{H^s}^{2\si-1}\|(1-\Pi_h)(|u(t)|^2 )\|_{L^2}.
\end{align*}
Like in the proof of Proposition~\ref{prop:CVsmooth}, we use the estimate
\begin{equation}\label{eq:est}
  \|\(1-\Pi_h\)f\|_{L^2} \le h^{s}\|f\|_{H^{s}},
\end{equation}
and since $H^s(\R^d)$ is an algebra,
\begin{equation*}
     \|R^h(u)\|_{L^\infty([0,T];L^2)}\lesssim
     h^{s}\|u\|_{L^\infty([0,T];H^s)}^{2\si+1}.
  \end{equation*}
To conclude the proof, we estimate $\nabla R^h(u)$ in $L^2(\R^d)$. We
compute
\begin{align*}
  \nabla R^h(u) & = \si |u|^{2\si-2}\(\(1-\Pi_h\)\(\nabla\(|u|^2\)\)\)
  u \\
& \quad + \si\( |u|^{2\si-2}-\(\Pi_h(|u|^2)\)^{\si-1}\) \Pi_h
\(\nabla\(|u|^2\)\) u\\
&\quad + \(|u|^{2\si}-\(\Pi_h\(|u|^2\)\)^{\si}\) \nabla u,
\end{align*}
where the second line is zero is $\si=1$. We estimate successively,
thaks to \eqref{eq:Youngh},
\begin{align*}
  \left\| |u|^{2\si-2}\(\(1-\Pi_h\)\(\nabla\(|u|^2\)\)\)
  u\right\|_{L^2}& \le
\|u\|_{L^\infty}^{2\si-1}\left\|\(1-\Pi_h\)\(|u|^2\)\right\|_{H^1},\\
\left\| \(|u|^{2\si}-\(\Pi_h\(|u|^2\)\)^{\si}\) \nabla
  u\right\|_{L^2}& \le
\|u\|_{L^\infty}^{2\si-2}\left\|\(1-\Pi_h\)\(|u|^2\)\right\|_{L^\infty}\|\nabla
u\|_{L^2} ,
\end{align*}
and, if $\si\ge 2$, 
\begin{align*}
\Big\|\( |u|^{2\si-2}-\(\Pi_h(|u|^2)\)^{\si-1}\) &\Pi_h
\(\nabla\(|u|^2\)\) u\Big\|_{L^2}\\
& \lesssim
\|u\|_{L^\infty}^{2\si-4}\left\|\(1-\Pi_h\)\(|u|^2\)\right\|_{L^2}
\left\|\nabla\(|u|^2\)\right\|_{L^2}\|u\|_{L^\infty}\\
& \lesssim
\|u\|_{L^\infty}^{2\si-2}\left\|\(1-\Pi_h\)\(|u|^2\)\right\|_{L^2}
\left\|\nabla u\right\|_{L^2}.
\end{align*}
Since we have $H^s(\R^d)\hookrightarrow L^\infty(\R^d)$, we end up
with
\begin{equation*}
  \|\nabla R^h(u)\|_{L^2}\lesssim
  \|u\|_{H^s}^{2\si-2}\left\|\(1-\Pi_h\)\(|u|^2\)\right\|_{H^1}. 
\end{equation*}
If $s>1$, \eqref{eq:est} yields, since in addition $s>d/2$,
\begin{equation*}
  \left\|\(1-\Pi_h\)\(|u|^2\)\right\|_{H^1}\lesssim
  h^{s-1}\left\||u|^2\right\|_{H^s} \lesssim h^{s-1}\|u\|_{H^s}^2.
\end{equation*}
If $s=1$ (a case which may occur only if $d=1$, since $s>d/2$), we
write
\begin{align*}
  \left\|\nabla \(1-\Pi_h\)\(|u|^2\)\right\|_{L^2}^2 \le
  \int_{|\xi|>1/h} \left|\F\( \nabla \(|u|^2\)\)(\xi)\right|^2d\xi.
\end{align*}
Now since $ \nabla
\(|u|^2\)=2\RE \bar u \nabla u$ and $u\in H^1(\R)\hookrightarrow
L^\infty(\R)$, $\nabla u\in L^2(\R)$, we
conclude thanks to the Dominated Convergence Theorem.
\end{proof}
This completes the proof of
Theorem~\ref{theo:CVstrichartz}, by choosing $(p_1,q_1)=(p,q)$ or
$(\infty,2)$.  
\end{proof}

\appendix

\section{Physical saturation of  the nonlinearity}
\label{sec:saturation}

Instead of cutting off the high frequencies, one may be tempted by
saturating the nonlinear potential, by replacing $|u|^2$ not by $\Pi(|u|^2)$
but by $f(|u|^2)$ where $f$ is smooth, equal to the identity near the
origin, and constant at infinity. Note also that a saturated
nonlinearity may be in better agreement with physical models (recall
however that \eqref{eq:nls} appears in rather different physical
contexts, such as quantum mechanics, optics, and even fluid mechanics), since
typically the power-like nonlinearity in \eqref{eq:nls} may stem from
a Taylor expansion; see e.g. \cite{Lan11,Sulem}. More precisely, let
$f\in C^\infty(\R;\R)$ such that
\begin{equation}\label{eq:f1}
  f(s) = \left\{
    \begin{aligned}
      s& \text{ if } 0\le s\le 1,\\
      2 & \text{ if } s\ge 2.
    \end{aligned}
\right.
\end{equation}
The analogue of the Fourier multiplier $\Pi_h$ is defined as
\begin{equation*}
  f_h\(|u|^2\) = \frac{1}{h}f\(h|u|^2\),
\end{equation*}
and we replace \eqref{eq:edp2} with
\begin{equation}
  \label{eq:satur}
  i\d_t u^h + P_h(D)u^h = \epsilon \(f_h(|u^h|^2\)^\si u^h,
\end{equation}
so the formal conservation of the $L^2$-norm still holds. 
We could also consider
\begin{equation}\label{eq:f2}
  f_h\(|u|^2\) =\frac{|u|^2}{1+h|u|^2}.
\end{equation}
In both cases, the main aspect to notice is that $f_h$ is bounded and $z\mapsto
f_h(|z|^2)^\si z$ is globally Lipschitzean. We infer the
analogue of Theorem~\ref{theo:gwp}, at least in the $L^2$ case.
\begin{proposition}
    Let $\si\in \N$, $\epsilon\in \{\pm 1\}$, $P:\R^d\to \R$ and
  $f$ given either by \eqref{eq:f1} or by \eqref{eq:f2}. 
  \begin{itemize}
  \item For any $u_0\in L^2(\R^d)$, \eqref{eq:satur} has a unique
    solution $u^h\in C(\R;L^2(\R^d))$ such that $u^h_{\mid t=0}=u_0$. Its
    $L^2$-norm is independent of time. 
\item The flow map $u_0\mapsto u^h$ is uniformly continuous from the
  balls in $L^2(\R^d)$ to $C(\R;L^2(\R^d))$. More precisely, for all
  $u_0,v_0\in L^2(\R^d)$, there exists $C$ depending on $\si$, $h$, 
  $\|u_0\|_{L^2}$ and $\|v_0\|_{L^2}$ such that for all $T>0$, 
  \begin{equation*}
    \|u^h-v^h\|_{L^\infty([-T,T];L^2(\R^d))}\le \|u_0-v_0\|_{L^2(\R^d)}e^{CT},
  \end{equation*}
where $u^h$ and $v^h$ denote the solutions to \eqref{eq:satur} with 
data $u_0$ and $v_0$, respectively. 
  \end{itemize}
\end{proposition}
Introduce
\begin{equation*}
  F_h(s) = \int_0^s f_h(y)^\si dy.
\end{equation*}
We check that the following conservation of energy holds:
\begin{equation*}
  \frac{d}{dt} \( \int_{\R^d} \overline u^h(t,x)P_h(D)u^h(t,x)dx +
  \epsilon \int_{\R^d} F_h\(|u(t,x)|^2\)dx\)=0. 
\end{equation*}
Proving the analogue of Proposition~\ref{prop:CVsmooth}  is easy in
the case \eqref{eq:f1}, since the last 
source term for the error $w^h$ is now 
\begin{equation*}
R^h(u)=  \(|u|^{2\si} - f_h\(|u|^2\)^\si\)u, 
\end{equation*}
and under the assumptions of Proposition~\ref{prop:CVsmooth}, $u\in
L^\infty([0,T]\times \R^d)$, so there exists $h_0>0$ such that for
$0<h\le h_0$,
\begin{equation*}
  |u(t,x)|^{2\si} = f_h\(|u(t,x)|^2\)^\si,\quad \forall (t,x)\in
  [0,T]\times \R^d.
\end{equation*}
Therefore, this source term simply vanishes for $h$ sufficiently
small. In the case \eqref{eq:f2}, we can use the relation
\begin{equation}\label{eq:sourcesatur}
 |R_h(u)|=   \left\lvert\( |u|^{2\si} -
   f_h\(|u|^2\)^\si\)u\right\rvert\lesssim
 \frac{h|u|^2}{1+h|u|^2}|u|^{2\si+1}, 
\end{equation}
and Schauder lemma to get a source
term which is $\O(h)$ in $H^s(\R^d)$, for $s>d/2$. 
\begin{proposition}\label{prop:CVsmooth2}
 Let $\si\in \N$. We assume that $P$ is such that there exist
 $\alpha,\beta\ge 0$ with $P_h(\xi) = -|\xi|^2 + 
    \O\(h^\alpha\<\xi\>^\beta\)$.
Denote by $u^h$ the solution to \eqref{eq:satur} with
 $P_h$ and $f_h$, such that $u^h_{\mid t=0}=u_{\mid t=0}$. Suppose
 that the solution to \eqref{eq:nls} satisfies $u \in
 L^\infty([0,T];H^{s+\beta})$, for some  
 $s>d/2$. 
 \begin{itemize}
 \item In the case \eqref{eq:f1},
 $    \|u-u^h\|_{L^\infty([0,T];H^{s})}\lesssim h^{\alpha}.$
\item In the case \eqref{eq:f2},
$     \|u-u^h\|_{L^\infty([0,T];H^{s})}\lesssim h^{\min(\alpha,1)}.$
 \end{itemize}
\end{proposition}
In the case \eqref{eq:f1}, proving an analogue to
Theorem~\ref{theo:CVstrichartz} seems to be more delicate though, and
we choose not to investigate this aspect here. On the other hand, in
the case \eqref{eq:f2}, using the estimate \eqref{eq:sourcesatur},
Strichartz estimates and H\"older inequalities with the ``standard''
Lebesgue exponents (in the same fashion as in the proof of
Theorem~\ref{theo:CVstrichartz}, see e.g. \cite{CazCourant}),  we
have, with steps similar to those presented in the proof of Theorem~\ref{theo:CVstrichartz}:
\begin{theorem}
  Let $\si\in \N$ and $T>0$. Let $u$ solve \eqref{eq:nls}, and
  consider $u^h$ solution to
  \begin{equation*}
   i\d_t u^h +\Delta u^h = \epsilon \(\frac{|u^h|^2}{1+h|u^h|^2}\)^\si
   u^h;\quad u^h_{\mid t=0}=u_0.
  \end{equation*}
$1.$ If $\si\le 2/d$,  and $u \in
L^\infty([0,T];L^2)\cap 
L^{(4\si+4)/d\si}([0,T];L^{2\si+2})$, then 
\begin{equation*}
     \|u-u^h\|_{L^\infty([0,T];L^2)}\Tend h 0 0 . 
  \end{equation*}
$2.$ Suppose that $\si=1$ and $d=3$.
\begin{itemize}
\item If $u,\nabla u \in
  L^\infty([0,T];L^2)\cap 
L^{8/d}([0,T];L^4)$, then
  \begin{equation*}
     \|u-u^h\|_{L^\infty([0,T];H^1)} \Tend h 0 0 . 
  \end{equation*}
\item If $u \in L^\infty([0,T];H^{s})$, with $s>3/2$,
then
 \begin{equation*}
    \|u-u^h\|_{L^\infty([0,T];H^{1})}\lesssim h.
  \end{equation*}
\end{itemize}
$3.$ Suppose that $\si\ge 1$ and $d\le 2$. 
 If $u \in L^\infty([0,T];H^{s})$, with $s\ge 1$ and $s>d/2$,
  then 
  \begin{equation*}
    \|u-u^h\|_{L^\infty([0,T];H^{1})}\lesssim h.
  \end{equation*}
\end{theorem}

 \end{document}